\documentclass[11pt,reqno]{amsart}
\usepackage{indentfirst, amssymb,amsmath,amsthm, mathrsfs,cite,graphicx,float}
\newtheorem*{theoA}{Theorem A}
\newtheorem*{theoB}{Theorem B}
\newtheorem*{theoC}{Theorem C}
\newtheorem*{theoD}{Theorem D}
\newtheorem*{theoE}{Theorem E}
\newtheorem*{theoF}{Theorem F}

\newtheorem{theo}{Theorem}[section]
\newtheorem{lem}{Lemma}[section]

\newtheorem{rem}{Remark}[section]

\newtheorem{ques}{Question}[section]

\newtheorem{open problem}{Open problem}[section]
\newcommand{\pa}{\partial}
\newcommand{\ol}{\overline}
\newcommand{\be}{\begin{equation}}
\newcommand{\ee}{\end{equation}}
\newcommand{\bs}{\begin{small}}
\newcommand{\es}{\end{small}}
\newcommand{\beas}{\begin{eqnarray*}}
\newcommand{\eeas}{\end{eqnarray*}}
\newcommand{\bea}{\begin{eqnarray}}
\newcommand{\eea}{\end{eqnarray}}
\renewcommand{\epsilon}{\varepsilon}
\numberwithin{equation}{section}

\begin{document}
\title[Bohr type inequality]
{Bohr type inequality for certain integral operators and Fourier transform on shifted disks}
\author[V. Allu, R. Biswas and R. Mandal]{ Vasudevarao Allu$^{*1}$, Raju Biswas$^2$ and Rajib Mandal$^2$}
\date{}
\maketitle
\let\thefootnote\relax
\footnotetext{Corresponding Author: V. Allu$^*$\\\indent avrao@iitbbs.ac.in\\[2mm]\indent R. Biswas\\\indent rajubiswasjanu02@gmail.com\\[2mm]\indent R. Mandal\\\indent rajibmathresearch@gmail.com\\[2mm]
$^1$ Indian Institute of Technology Bhubaneswar, School of Basic Science, Bhubaneswar-752050, Odisha, India.\\
$^2$ Department of Mathematics, Raiganj University, Raiganj, West Bengal-733134, India.\\}

\footnotetext{Type set by \AmS -\LaTeX}
\begin{abstract}
In this paper, we derive the sharp Bohr type inequality for the Ces\'aro operator, Bernardi integral operator, and discrete Fourier transform acting on the class of bounded analytic functions defined on shifted disks
\beas \Omega_{\gamma}=\left\{z\in\mathbb{C}:\left|z+\frac{\gamma}{1-\gamma}\right|<\frac{1}{1-\gamma}\right\}\quad\text{for}\quad\gamma\in[0,1).\eeas
\end{abstract}
\section{Introduction and Preliminaries}
Let $\mathbb{D}_{\rho}(a)=\{z\in\mathbb{C}:|z-a|<\rho\}$ be an open disk center at $a$ with radius $\rho$ and let $\mathbb{D} := \mathbb{D}_1(0)$ be the open unit disk in $\mathbb{C}$. 
Let $\Omega$ be a simply connected domain such that $\mathbb{D}\subseteq \Omega$ and $\mathcal{H}(\Omega)$ denotes the class of analytic functions on $\Omega$. Let $\mathcal{B}(\Omega)=\left\{f \in \mathcal{H}(\Omega): f(\Omega)\subseteq \ol{\mathbb{D}}\right\}$, and the Bohr radius (see \cite{10}) for the class $\mathcal{B}(\Omega)$ is defined to be the number $B_\Omega\in (0, 1)$ such that 
 \beas B_\Omega=\sup\left\{\rho\in (0,1):M_f(\rho)\leq 1\;\text{for} \;f(z)=\sum_{n=0}^{\infty}\alpha_nz^n\in\mathcal{B}(\Omega),z\in\mathbb{D}\right\},\eeas 
where $M_f(\rho)=\sum_{n=0}^{\infty}|\alpha_n|\rho^n$ is the majorant series associated with the analytic functions $f\in\mathcal{B}(\Omega)$ in $\mathbb{D}$. 
It is well known that $B_{\mathbb{D}} = 1/3$ if $\Omega=\mathbb{D}$. This is described in the following manner:
\begin{theoA} (The Classical Bohr Theorem) If $f\in\mathcal{B}(\mathbb{D})$, then $M_f (\rho) \leq 1$ for $0 \leq \rho \leq 1/3$. The number $1/3$ is best possible.
\end{theoA}
Note that, the inequality $M_f (\rho) \leq 1$, where $f\in\mathcal{B}(\mathbb{D})$, fails to hold for $\rho > 1/3$.
For $\varphi_r(z) = (r - z)/(1 - rz)$, it is easy to see that $M_{\varphi_r} (\rho)>1$ if, and only if, $\rho>1/(1+2r)$. This shows that $1/3$ is optimal for $r\to1$.\\[2mm]
\indent \textrm{Theorem A} was actually obtained by H. Bohr in 1914 (see \cite{5}) for $\rho \leq 1/6$, but  later Weiner, Riesz, and Schur \cite{8a} have independently improved 
this result to a value of $1/3$. 
This result has also been established by Sidon \cite{28} and Tomi\'c \cite{30}. Over the past two decades, a considerable amount of research has been conducted on Bohr type 
inequalities. For an in-depth investigation on Bohr radius and Bohr inequality, see \cite{1,2,301,302,303,304,305,4,8,12,306,307,14,16,308,R1} and the references therein. 
Boas and Khavinson \cite{7} further developed the concept of the Bohr radius, especially for several complex variables, and introduced the concept of the multidimensional Bohr 
radius. We refer to \cite{3,V1,V2,21a} for an in-depth study on multidimensional Bohr radius.\\[2mm]
\indent For $\gamma\in[0,1)$, we consider the open disk $\Omega_{\gamma}$ defined by
\bea\label{a1} \Omega_{\gamma}=\left\{z\in\mathbb{C}:\left|z+\frac{\gamma}{1-\gamma}\right|<\frac{1}{1-\gamma}\right\}.\eea
Note that the unit disk $\mathbb{D}$ is always a subset of  $\Omega_{\gamma}$. 
Figure \ref{fig1} illustrates the pictures of the circles $C_\gamma : \left|z+\gamma/(1-\gamma)\right|=1/(1-\gamma)$  for different values of $\gamma \in [0, 1)$.
\begin{figure}[H]
\centering
\includegraphics[scale=0.8]{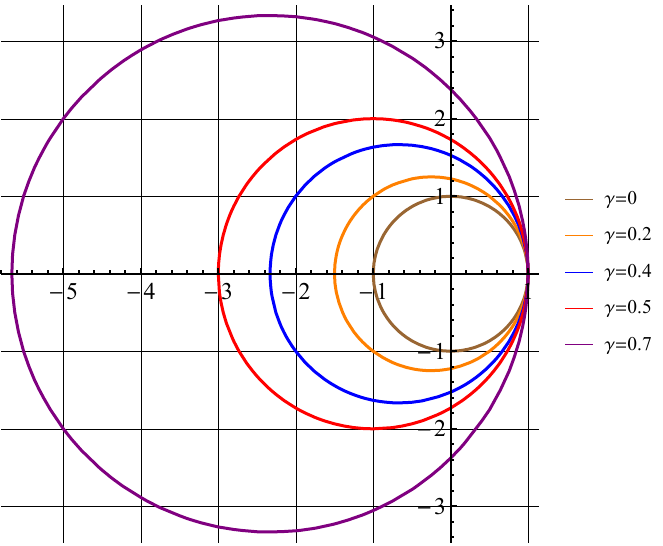}
\caption{The graphs of $C_{\gamma}$ when $\gamma=0,0.2,0.4,0.5,0.7$}
\label{fig1}
\end{figure}
Besides the Bohr radius, there is another concept known as the Rogosinski radius \cite{21,26}. It is described as follows:
Let $S_N(z):=\sum_{n=0}^{N-1}a_nz^n$ be the partial sum of $f\in\mathcal{B}(\mathbb{D})$ defined by $f(z)=\sum_{n=0}^\infty a_nz^n$. Then, $|S_N(z)|<1$ for all 
$N\geq 1$ in the disk $|z|<1/2$. Here $1/2$ is sharp, known as Rogosinski radius.
Kayumov and Ponnusamy \cite{13} have introduced the Bohr-Rogosinski sum $R_N^f(z)$, which is motivated by the Rogosinski radius.  It is defined as follows:
\beas R_N^f(z):=|f(z)|+\sum_{n=N}^{\infty}|a_n||z|^n.\eeas
It is evident that $|S_N(z)|=\left|f(z)-\sum_{n=N}^{\infty}a_nz^n\right|\leq R_N^f(z)$. 
Kayumov and Ponnusamy \cite{13} defined the Bohr-Rogosinski radius as the largest number $r\in(0,1)$ such that $R_N^f(z)\leq 1$ for $|z| < r$. 
We refer to \cite{12,15,18} and the references therein for an in-depth study of the Bohr-Rogosinski radius.\\[2mm]
The following question is expected in such a context.
\begin{ques} Is it possible to extend the Bohr type inequality for certain complex integral operators and integral transforms that are defined on different function spaces?
\end{ques}
The idea was initially proposed in the context of unit disk $\mathbb{D}$ for the classical Ces\'aro operator by Kayumov {\it et al.} \cite{17,18,19} and for the Bernardi integral operator by Kumar and Sahoo \cite{20}.
The objective of establishing the Bohr type and Bohr-Rogosinski type inequalities for the Bernardi integral operator, Ces\'aro operator, and its various generalizations on the family $\mathcal{B}(\mathbb{D})$ has been the subject of research. Additionally, extensive research has been conducted regarding the boundedness and compactness of the Ces\'aro operator in different function spaces. In the classical setting, for an analytic function $f(z)=\sum_{n=0}^\infty a_nz^n$ on the unit disk $\mathbb{D}$, the Ces\'aro operator \cite{8,11,29} is defined as
\bea \mathcal{C}f(z):=\sum_{n=0}^\infty \frac{1}{n+1}\left(\sum_{k=0}^n a_k\right) z^n=\int_{0}^1\frac{f(t z)}{1-t z} d t.\eea
It is not difficult to show that for $f\in\mathcal{B}(\mathbb{D})$,
\beas \vert\mathcal{C}f(z)\vert=\left\vert\sum_{n=0}^\infty \frac{1}{n+1}\left(\sum_{k=0}^n |a_k|\right) z^n\right\vert\leq \frac{1}{\rho}\log\frac{1}{1-\rho}\quad\text{for}\quad\vert z\vert=\rho.\eeas 
For the Ces\'aro operator, Kayumov {\it et al.} \cite{17} established the following Bohr type inequality.
\begin{theoB}\cite[Theorem 1, p. 616]{17} If $f\in\mathcal{B}(\mathbb{D})$ and $f(z)=\sum_{n=0}^\infty a_n z^n$, then 
\beas \mathcal{C}_f(\rho)=\sum_{n=0}^\infty \frac{1}{n+1}\left(\sum_{k=0}^n |a_k|\right) \rho^n\leq  \frac{1}{\rho}\log\frac{1}{1-\rho}\quad\text{for}\quad\vert z\vert=\rho\leq R,\eeas
where $R=0.5335...$ is the positive root of the equation 
\beas 2x=3(1-x)\log \frac{1}{1-x}.\eeas 
The number $R$ is the best possible.\end{theoB}
For an analytic function $f(z)=\sum_{n=m}^\infty a_n z^n$ on the unit disk $\mathbb{D}$, the Bernardi integral operator \cite[p. 11]{21c} is defined as
\bea\label{aa1} \mathcal{L}_\beta f(z):=\frac{1+\beta}{z^\beta}\int_0^z f(\eta) \eta^{\beta-1} d \eta=(1+\beta)\sum_{n=m}^\infty \frac{a_n}{\beta+n} z^n, \eea
where $\beta>-m$ and $m\geq 0$ is an integer. From (\ref{aa1}), it is worth to mentioning that for $f\in\mathcal{B}(\mathbb{D})$ with $|z|=\rho\in[0,1)$, we have
\beas |\mathcal{L}_\beta f(z)|=\left|(1+\beta)\sum_{n=m}^\infty \frac{a_n}{\beta+n} z^n\right|\leq (1+\beta)\frac{\rho^m}{m+ \beta},\eeas
which is equivalent to the following inequality
\beas \left|\sum_{n=m}^\infty \frac{a_n}{\beta+n} z^n\right|\leq \frac{\rho^m}{m+ \beta}.\eeas
In 2021, Kumar and Sahoo \cite{20} established the following Bohr type inequality with regard to the Bernardi integral operator.
\begin{theoC}\cite{20} Let $\beta>-m$. If $f(z)=\sum_{n=m}^\infty a_n z^n\in\mathcal{B}(\mathbb{D})$, then 
\beas \sum_{n=m}^\infty \frac{|a_n|}{\beta+n} \rho^n\leq \frac{\rho^m}{m+ \beta}\quad\text{for}\quad |z|=\rho\leq R(\beta),\eeas
where $R(\beta)$ is the positive root of the equation 
\beas\frac{x^m}{m+\beta}-2\sum_{n=m+1}^\infty \frac{x^n}{n+\beta}=0.\eeas 
The number $R(\beta)$ cannot be improved.\end{theoC}
In 2023, Allu and Ghosh \cite{5} obtained the following sharp Bohr type inequality for the Ces\'aro operator and Bernardi integral operator for functions in the class $\mathcal{B}(\Omega_{\gamma})$ with $f(z)=\sum_{n=0}^\infty a_n z^n$ in $\mathbb{D}$.
\begin{theoD}\cite{5}
For $0\leq \gamma<1$, let  $f\in\mathcal{B}(\Omega_{\gamma})$ with $f(z)=\sum_{n=0}^\infty a_n z^n$ in $\mathbb{D}$. Then, we have 
\beas \mathcal{C}_f(\rho)=\sum_{n=0}^\infty \frac{1}{n+1}\left(\sum_{k=0}^n |a_k|\right) \rho^n\leq \frac{1}{\rho}\log\frac{1}{1-\rho}\quad\text{for}\quad |z|=\rho\leq R_\gamma,\eeas
where $R_\gamma$ is the positive root of the equation 
\beas (3+\gamma)(1-x)\log\frac{1}{1-x}=2x.\eeas 
The number $R_\gamma$ is the best possible. \end{theoD}
\begin{theoE}\cite{5}
For $0\leq \gamma<1$, let  $f\in\mathcal{B}(\Omega_{\gamma})$ with $f(z)=\sum_{n=0}^\infty a_n z^n$ in $\mathbb{D}$. Then for $\beta>0$
\beas \sum_{n=0}^\infty \frac{|a_n|}{n+\beta}\rho^n\leq \frac{1}{\beta}\quad\text{for}\quad |z|=\rho\leq R_{\gamma,\beta},\eeas
where $R_{\gamma,\beta}$ is the positive root of the equation 
\beas\frac{1}{\beta}=\frac{2}{1+\gamma}\sum_{n=1}^\infty \frac{\rho^n}{n+\beta}.\eeas 
The number $R_{\gamma,\beta}$ is the best possible.
\end{theoE}
It is observed that the Bohr type inequality presented in \textrm{Theorems D} and \textrm{E} imposes a restriction on analytic functions defined on the simply connected domain 
$\Omega_{\gamma}$ to the open unit disk $\Bbb{D}$.\\[2mm]
\indent Let $\left\{x_n\right\}_{n=0}^{N-1}$ be a sequence of complex numbers. The discrete Fourier transform (see \cite[Chapter 6]{100}) is defined as 
\beas \mathcal{F}(x_n)=\sum_{n=0}^{N-1} x_n e^{-2\pi i n k/N}. \eeas 
For $f(z)=\sum_{n=0}^\infty a_n z^n\in\mathcal{B}(\mathbb{D})$, we perform the discrete Fourier transform on the coefficients $a_k$ from $k=0$ to $n$, which gives
\beas \mathcal{F}[f](z)=\sum_{n=0}^\infty \left(\sum_{k=0}^n a_k e^{-2\pi i n k/(n+1)}\right) z^n.\eeas
To obtain a Bohr type inequality in the unit disk $\Bbb{D}$, we denote the majorant series of $\mathcal{F}[f](z)$ as
\beas \mathcal{F}_f(\rho):=\sum_{n=0}^\infty \left(\sum_{k=0}^n |a_k|\right) \rho^n, \quad\text{where}\quad |z|=\rho<1.\eeas
For functions in the class $\mathcal{B}(\mathbb{D})$, in 2024, Ong and Ng \cite{101} derived the following Bohr type inequality with regard to the discrete Fourier transform.
\begin{theoF}\cite{101}
Let $f(z)=\sum_{n=0}^\infty a_n z^n\in\mathcal{B}(\mathbb{D})$. Then, 
\beas \mathcal{F}_f(\rho)\leq \frac{1}{1-\rho}\quad\text{for}\quad \rho\leq \frac{1}{3}.\eeas
The constant $1/3$ cannot be improved.
\end{theoF}
\section{Main results}
Motivated by Fournier and Ruscheweyh \cite{10}, in this paper, we define the Bohr radius for functions in the class $\mathcal{B}(\Omega_{\gamma})$. It is defined to be the 
number 
$R_{\Omega_{\gamma}}\in (0, 1)$ such that 
\bs\beas R_{\Omega_{\gamma}}=\sup\left\{\rho\in (0,1):M_f(\rho)\leq 1\;\text{for} \;f(z)=\sum_{n=0}^{\infty}\alpha_n\left(z+\frac{\gamma}{1-\gamma}\right)^n\in\mathcal{B}(\Omega_{\gamma}), z\in\Omega_{\gamma}\right\},\eeas \es
where $M_f(\rho)=\sum_{n=0}^{\infty}|\alpha_n|\left(\rho/(1-\gamma)\right)^n$ with $|\gamma+(1-\gamma)z|=\rho$, is the majorant series associated with the analytic function $f\in\mathcal{B}(\Omega_{\gamma})$. It is well known that $R_{\mathbb{D}} = 1/3$ if $\Omega_{\gamma}=\mathbb{D}$.\\[2mm]
\indent 
The objective of this paper is to derive sharply Bohr type inequality for the Ces\'aro operator, Bernardi integral operator, and discrete Fourier transform for functions in the class 
$\mathcal{B}(\Omega_{\gamma})$ without imposing the restriction to the unit disk $\Bbb{D}$.\\[2mm]
\indent The following are key lemmas of this paper and will be used for the proof of the main results.
\begin{lem}\cite{26a}\label{lem1} For $f\in\mathcal{B}(\mathbb{D})$, then we have 
\beas \frac{\left|f^{(n)}(\alpha)\right|}{n!}\leq \frac{1-|f(\alpha)|^2}{(1-|\alpha|)^{n-1}(1-|\alpha|^2)}\;\;\text{for each}\;\; n\geq 1\;\;\text{and}\;\; \alpha\in\mathbb{D}.\eeas\end{lem}
\begin{lem}\label{lem2}Let $f$ be analytic in $\Omega_{\gamma}$, bounded by $1$ with the series expansion $f(z)=\sum_{n=0}^\infty a_n\left(z+\frac{\gamma}{1-\gamma}\right)^n$ in $\Omega_{\gamma}$. Then $|a_n|\leq (1-\gamma)^n(1-|a_0|^2)$ for $n\geq 1$.\end{lem}
\begin{proof} Let $\Phi : \mathbb{D}\to\Omega_{\gamma}$ be a function defined by $\Phi(z)=(z-\gamma)/(1-\gamma)$. Since $f : \Omega_{\gamma}\to\ol{\mathbb{D}}$, 
the composition $g=f\circ \Phi$ is analytic in $\mathbb{D}$. Thus 
\beas g(z)=f\left(\frac{z-\gamma}{1-\gamma}\right)=\sum_{n=0}^\infty \frac{a_n}{(1-\gamma)^n}z^n\eeas
with
\beas g(0)=a_0\quad\text{and}\quad a_n=\frac{g^{(n)}(0)}{n!}(1-\gamma)^n.\eeas
In view of \textrm{Lemma \ref{lem1}}, we have 
\beas|a_n|\leq (1-\gamma)^n\left(1-|g(0)|^2\right)=(1-\gamma)^n\left(1-|a_0|^2\right).\eeas\end{proof}
In the following, we derive a sharp Bohr type inequality for the Ces\'aro operator applied to functions belonging to the class $\mathcal{B}(\Omega_{\gamma})$ of bounded analytic functions.
\begin{theo}\label{Th1}
For $0\leq \gamma<1$, let  $f\in\mathcal{B}(\Omega_{\gamma})$ with $f(z)=\sum_{n=0}^\infty a_n\left(z+\frac{\gamma}{1-\gamma}\right)^n$ in $\Omega_{\gamma}$. Then, we have 
\beas \mathcal{C}_f(\rho)=\sum_{n=0}^\infty \frac{1}{n+1}\left(\sum_{k=0}^n\frac{|a_k|}{(1-\gamma)^k}\right) \rho^n\leq \frac{1}{\rho}\log\frac{1}{1-\rho}\;\;\text{for}\;|\gamma+(1-\gamma)z|=\rho\leq \rho_0,\eeas
where $\rho_0(\approx 0.533589)$ is the positive root of the equation $3(1-\rho)\log(1-\rho)+2\rho=0$. The number $\rho_0$ is the best possible.
\end{theo}
\begin{rem}
Since $\Omega_{\gamma}=\mathbb{D}$ for $\gamma=0$, it follows that \textrm{Theorem A} is a direct consequence of \text{Theorem \ref{Th1}} when $\gamma=0$.\end{rem}
In the following, we derive a sharp Bohr type inequality for the Bernardi integral operator applied to functions belonging to the class $\mathcal{B}(\Omega_{\gamma})$ of bounded analytic functions.
\begin{theo}\label{Th2}
For $0\leq \gamma<1$, let  $f\in\mathcal{B}(\Omega_{\gamma})$ with $f(z)=\sum_{n=0}^\infty a_n\left(z+\frac{\gamma}{1-\gamma}\right)^n$ in $\Omega_{\gamma}$. Then for $\beta>0$, we have 
\beas \sum_{n=0}^\infty \frac{|a_n|}{(n+\beta)(1-\gamma)^n}\rho^n\leq \frac{1}{\beta}\quad\text{for}\quad|\gamma+(1-\gamma)z|=\rho\leq \rho_1,\eeas
where $\rho_1$ is the positive root of the equation $1/\beta - 2\sum_{n=1}^\infty\left( \rho^n/(n+\beta)\right)=0$. The number $\rho_1$ is the best possible.
\end{theo}
In the following, we derive a sharp Bohr type inequality the discrete Fourier transform applied to functions belonging to the class $\mathcal{B}(\Omega_{\gamma})$ of bounded analytic functions.
\begin{theo}\label{Th3}
For $0\leq \gamma<1$, let  $f\in\mathcal{B}(\Omega_{\gamma})$ with $f(z)=\sum_{n=0}^\infty a_n\left(z+\frac{\gamma}{1-\gamma}\right)^n$ in $\Omega_{\gamma}$. Then, we have 
\beas\sum_{n=0}^\infty \left(\sum_{k=0}^n \frac{|a_k|}{(1-\gamma)^n}\right) \rho^n\leq \frac{1}{1-\rho}\quad\text{for}\quad |\gamma+(1-\gamma)z|=\rho\leq \frac{1}{3},\eeas
The number $1/3$ is the best possible.\end{theo}
\begin{rem} Since for $\gamma=0$, the domain $\Omega_{\gamma}$ reduces to the unit disk $\mathbb{D}$, it follows that \textrm{Theorem F} is a direct consequence of \text{Theorem \ref{Th3}} when $\gamma=0$.\end{rem}
\section{proofs of the main results}
\begin{proof}[{\bf Proof of Theorem \ref{Th1}}]
Let $f(z)$ be analytic on $\Omega_{\gamma}$ with $|f(z)|\leq 1$. In view of \textrm{Lemma \ref{lem2}}, we have $|a_n|\leq (1-\gamma)^n (1-|a_0|^2)$ for $n\geq 1$. Then
\beas\mathcal{C}_f(\rho)&=&\sum_{n=0}^\infty \frac{1}{n+1}\left(\sum_{k=0}^n \frac{|a_k|}{(1-\gamma)^k}\right) \rho^n\\[2mm]
&=& |a_0|\left(1+\frac{\rho}{2}+\frac{\rho^2}{3}+\cdots\right)+\sum_{n=0}^\infty \frac{1}{n+1}\left(\sum_{k=1}^n \frac{|a_k|}{(1-\gamma)^k}\right) \rho^n\\[2mm]
&=&-\frac{|a_0|}{\rho}\log(1-\rho)+\sum_{n=0}^\infty \frac{1}{n+1}\left(\sum_{k=1}^n \frac{|a_k|}{(1-\gamma)^k}\right) \rho^n.\eeas
Let $|a_0|=a\in[0,1]$. In view of \textrm{Lemma \ref{lem2}}, we have
\beas \mathcal{C}_f(\rho)&\leq& -\frac{a}{\rho}\log(1-\rho)+\sum_{n=0}^\infty \frac{1}{n+1}\left(\sum_{k=1}^n\left(1-a^2\right)\right) \rho^n\\[2mm]
&=&-\frac{a}{\rho}\log(1-\rho)+\left(1-a^2\right)\sum_{n=0}^\infty \frac{n}{n+1}\rho^n\\[2mm]
&=&-\frac{a}{\rho}\log(1-\rho)+\left(1-a^2\right) \left(\frac{1}{1-\rho}+\frac{1}{\rho}\log(1-\rho)\right)\\[2mm]&=&\varphi_1(a,\rho).\eeas
Differentiating partially $\varphi_1(a,\rho)$ with respect to $a$ twice, we have   
\beas&&\frac{\pa }{\pa a}\varphi_1(a,\rho)=-\frac{\log(1-\rho)}{\rho} - 2 a\left(\frac{1}{1-\rho}+\frac{1}{\rho}\log(1-\rho)\right)\\[2mm]\text{and}
&&\frac{\pa^2 }{\pa a^2}\varphi_1(a,\rho)=-2 \left(\frac{1}{1-\rho}+\frac{1}{\rho}\log(1-\rho)\right)\leq 0\quad\text{for}\quad\rho\in[0,1)\eeas
and it's shown in Figure \ref{fig2}.
\begin{figure}[H]
\centering
\includegraphics[scale=0.6]{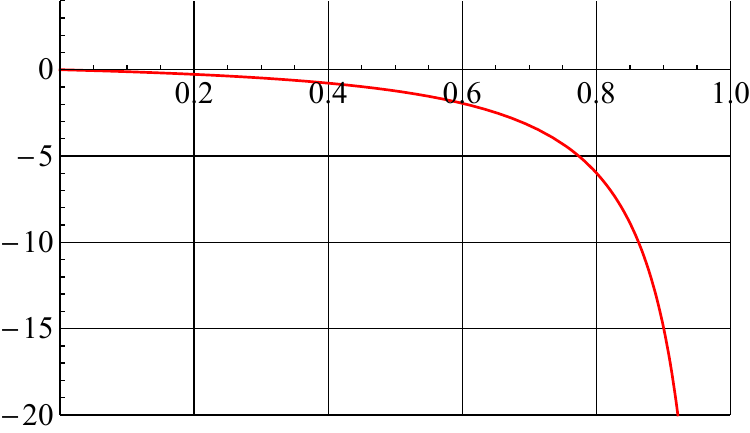}
\caption{The graph of $-2 \left(1/(1-\rho)+\log(1-\rho)/\rho\right)$ for $\rho\in[0,1)$}
\label{fig2}
\end{figure}
Therefore, $\frac{\pa }{\pa a} \varphi_1(a,\rho)$ is a monotonically decreasing function of $a\in[0,1]$ and it follows that 
\beas\frac{\pa}{\pa a}\varphi_1(a,\rho)\geq \frac{\pa }{\pa a}\varphi_1(1,\rho)=\frac{-2\rho- 3(1-\rho)\log(1-\rho)}{(1-\rho)\rho}\geq 0\quad\text{for}\quad \rho\leq \rho_0,\eeas
where $\rho_0(\approx 0.533589)$ is the positive root of the equation $3(1-\rho)\log(1-\rho)+2\rho=0$, as illustrated in Figure \ref{fig3}.
\begin{figure}[H]
\centering
\includegraphics[scale=0.7]{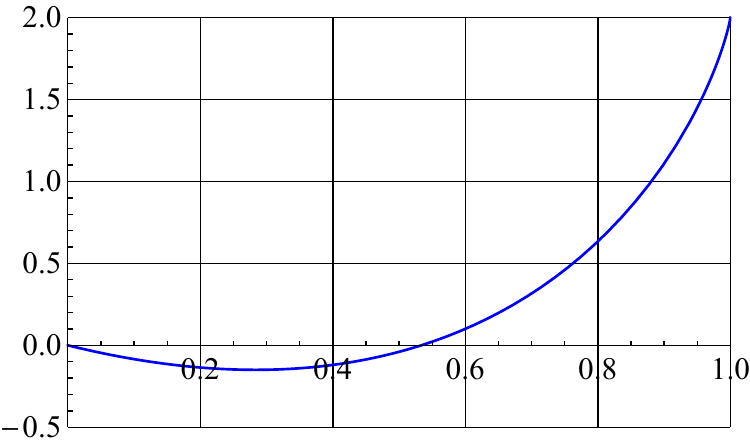}
\caption{The graph of $3(1-\rho)\log(1-\rho)+2\rho$ for $\rho\in[0,1)$}
\label{fig3}
\end{figure}
Therefore $\varphi_1(a,\rho)$ is a monotonically increasing function of $a\in[0,1]$ and it follows that 
\beas \varphi_1(a,\rho)\leq \varphi_1(1,\rho)=-\frac{1}{\rho}\log(1-\rho)=\frac{1}{\rho}\log\frac{1}{(1-\rho)}\quad\text{for}\quad\rho\leq \rho_0.\eeas\\[2mm]
\indent To prove the sharpness of the result, we consider the function $f_1(z)$ in $\Omega_{\gamma}$ such that $f_1=\psi\circ\Phi_1$, where $\Phi_1 : \Omega_{\gamma}\to\mathbb{D}$ defined by $\Phi_1(z)=\gamma+(1-\gamma)z$ and $\psi : \mathbb{D}\to\mathbb{D}$ defined by $\psi(z)=(a-z)/(1-az)$ for $a\in(0,1)$ and $\gamma\in[0,1)$. Thus, we have
\bea\label{b3} f_1(z)
=\frac{a-(1-\gamma)\left(z+\frac{\gamma}{1-\gamma}\right)}{1-a(1-\gamma)\left(z+\frac{\gamma}{1-\gamma}\right)}
=A_0-\sum_{n=1}^\infty A_n \left(z+\frac{\gamma}{1-\gamma}\right)^n\;\text{for}\;z\in\Omega_{\gamma},\eea	
where $A_0=a$ and $A_n=a^{n-1}(1-a^2)(1-\gamma)^n$ for $n\geq 1$.
The Ces\'aro operator on $f_1$ shows that 
\beas \mathcal{C}_{f_1}(r)&=&\sum_{n=0}^\infty \frac{1}{n+1}\left(\sum_{k=0}^n \frac{|A_k|}{(1-\gamma)^k}\right) \rho^n\\[2mm]
&=&-\frac{a}{\rho}\log(1-\rho)+\sum_{n=0}^\infty \frac{1}{n+1}\left(\sum_{k=1}^n a^{k-1}(1-a^2)\right) \rho^n\\[2mm]
&=&-\frac{1}{\rho}\log(1-\rho)-\frac{2a}{\rho}\log(1-\rho)+\frac{(1+a)}{a\rho}\log(1-a\rho).\eeas
As $2/(1-\rho)+3\log(1 - \rho)/\rho>0$ for $\rho>\rho_0$, we can write $\mathcal{C}_{f_1}(\rho)$ as 
\bea\label{b1}
\mathcal{C}_{f_1}(\rho)&=&-\frac{1}{\rho}\log(1-\rho)+(1-a)\left(\frac{2}{1-\rho}+\frac{3}{\rho}\log(1 - \rho)\right)-(3-a)\frac{1}{\rho}\log(1-\rho)\nonumber\\[2mm]
&&-\frac{2(1-a)}{1-\rho}+\frac{(1+a)}{a\rho}\log(1-a\rho)\nonumber\\[2mm]
&=&-\frac{1}{\rho}\log(1-\rho)+(1-a)\left(\frac{2}{1-\rho}+\frac{3}{\rho}\log(1 - \rho)\right)+G_1(a,\rho),\eea
where 
\beas G_1(a,\rho)=-(3-a)\frac{1}{\rho}\log(1-\rho)-\frac{2(1-a)}{1-\rho}+\frac{(1+a)}{a\rho}\log(1-a\rho).\eeas
In order to express $G_1(a,\rho)$ in its summation form, we have
\bea\label{b2}G_1(a,\rho)&=&(3-a)\sum_{n=1}^\infty \frac{\rho^{n-1}}{n}-2(1-a)\sum_{n=1}^\infty \rho^{n-1}-(1+a)\sum_{n=1}^\infty \frac{(a\rho)^{n-1}}{n}\nonumber\\[2mm]
&=&\sum_{n=1}^\infty \left(\frac{3-a}{n}-2(1-a)-(1+a)\frac{a^{n-1}}{n}\right)\rho^{n-1}\nonumber\\[2mm]
&=&O\left((1-a)^2\right)\quad\text{as}\quad a\to1^-.\eea
From (\ref{b1}) and (\ref{b2}), we conclude that $\rho_0$ cannot be improved. This completes the proof.
\end{proof}
\begin{proof}[{\bf Proof of Theorem \ref{Th2}}]
Let $f(z)$ be analytic on $\Omega_{\gamma}$ with $|f(z)|\leq 1$. Let $|a_0|=a\in[0,1]$. In view of \textrm{Lemma \ref{lem2}}, we have
\beas \sum_{n=0}^\infty \frac{|a_n|}{(n+\beta)(1-\gamma)^n}\rho^n&\leq& \frac{a}{\beta}+\sum_{n=1}^\infty \frac{(1-a^2)}{(n+\beta)}\rho^n\\[2mm]
&=& \frac{a}{\beta}+(1-a^2)\sum_{n=1}^\infty \frac{\rho^n}{(n+\beta)}\\[2mm]&=&\varphi_2(a,\rho).\eeas
Differentiating partially $\varphi_2(a,\rho)$ with respect to $a$ twice, we have   
\beas\frac{\pa }{\pa a}\varphi_2(a,\rho)=\frac{1}{\beta} - 2 a\sum_{n=1}^\infty \frac{\rho^n}{(n+\beta)}\quad\text{and}\quad\frac{\pa^2 }{\pa a^2}\varphi_2(a,\rho)=-2 \sum_{n=1}^\infty \frac{\rho^n}{(n+\beta)}\leq 0\eeas
for $\rho\in[0,1)$.
Therefore, $\frac{\pa}{\pa a} \varphi_2(a,\rho)$ is a monotonically decreasing function of $a$ and this gives that 
\beas \frac{\pa}{\pa a}\varphi_2(a,\rho)\geq \frac{\pa}{\pa a}\varphi_2(1,\rho)=\frac{1}{\beta} - 2\sum_{n=1}^\infty \frac{\rho^n}{(n+\beta)}\geq 0\quad\text{for}\quad\rho\leq \rho_1,\eeas
where $\rho_1$ is the positive root of the equation 
\beas \frac{1}{\beta} - 2\sum_{n=1}^\infty \frac{\rho^n}{(n+\beta)}=0.\eeas
Thus, for all $a\in[0,1]$, we have 
\beas \sum_{n=0}^\infty \frac{|a_n|}{(n+\beta)(1-\gamma)^n}\rho^n\leq \varphi_2(a,\rho)\leq \varphi_2(1,\rho)=\frac{1}{\beta}\quad\text{for}\quad\rho\leq \rho_1.\eeas\\[2mm]
\indent To prove the sharpness of the result, we consider the function $f_2(z)$ in $\Omega_{\gamma}$ such that $f_2=\psi\circ\Phi_2$, where $\Phi_2 : \Omega_{\gamma}\to\mathbb{D}$ defined by $\Phi_2(z)=\gamma+(1-\gamma)z$ and $\psi : \mathbb{D}\to\mathbb{D}$ defined by $\psi(z)=(a-z)/(1-az)$ for $a\in(0,1)$ and $\gamma\in[0,1)$.
Therefore,
\beas f_2(z)=A_0-\sum_{n=1}^\infty A_n \left(z+\frac{\gamma}{1-\gamma}\right)^n\quad\text{for}\quad a\in(0,1), \gamma\in[0,1)\quad\text{and}\quad z\in\Omega_{\gamma},\eeas	
where $A_0=a$ and $A_n=a^{n-1}(1-a^2)(1-\gamma)^n$ for $n\geq 1$. Thus,
\bea\label{b4}\sum_{n=0}^\infty \frac{|A_n|}{(n+\beta)(1-\gamma)^n}\rho^n&=&\frac{a}{\beta}+\sum_{n=1}^\infty \frac{|A_n|}{(n+\beta)(1-\gamma)^n}\rho^n\nonumber\\[2mm]
&=&\frac{1}{\beta}+(1-a)\left(2\sum_{n=1}^\infty \frac{\rho^n}{(n+\beta)}-\frac{1}{\beta} \right)+G_2(a,\rho),\eea
where
\beas G_2(a,\rho)&=&-2(1-a)\sum_{n=1}^\infty \frac{\rho^n}{(n+\beta)}+(1-a^2)\sum_{n=1}^\infty \frac{a^{n-1}}{(n+\beta)}\rho^n\\[2mm]
&=&\sum_{n=1}^\infty\frac{(1-a^2)a^{n-1}-2(1-a)}{(n+\beta)}\rho^n=O((1-a)^2)\quad\text{as}\quad a\to 1^-.\eeas 
Further, the quantity 
\beas 2\sum_{n=1}^\infty \frac{\rho^n}{(n+\beta)}-\frac{1}{\beta}>0\quad\text{for}\quad\rho>\rho_1.\eeas
From (\ref{b4}), we conclude that $\rho_1$ cannot be improved. This completes the proof.
\end{proof}
\begin{proof}[{\bf Proof of Theorem \ref{Th3}}]
Let $f(z)$ be analytic on $\Omega_{\gamma}$ with $|f(z)|\leq 1$. Let $|a_0|=a\in[0,1]$. In view of \textrm{Lemma \ref{lem2}}, we have 
\beas \sum_{n=0}^\infty \left(\sum_{k=0}^n \frac{|a_k|}{(1-\gamma)^n}\right) \rho^n&\leq& \sum_{n=0}^\infty a\rho^n+ \sum_{n=0}^\infty \left(\sum_{k=1}^n \frac{|a_k|}{(1-\gamma)^n}\right) \rho^n\\[2mm]
&\leq& \frac{a}{1-\rho}+(1-a^2)\sum_{n=0}^\infty n\rho^n\\
&=&\frac{a}{1-\rho}+(1-a^2)\frac{\rho}{(1-\rho)^2}\\&=&\varphi_3(a,\rho).\eeas
Differentiating partially $\varphi_3(a,\rho)$ with respect to $a$ twice, we have   
\beas\frac{\pa }{\pa a}\varphi_3(a,\rho)=\frac{1}{1-\rho} - \frac{2 a\rho}{(1-\rho)^2}\quad\text{and}\quad\frac{\pa^2 }{\pa a^2}\varphi_3(a,\rho)=-\frac{2\rho}{(1-\rho)^2}\leq 0\quad\text{for}\quad \rho\in[0,1).\eeas
Therefore, $\frac{\pa }{\pa a}\varphi_3(a,\rho)$ is a monotonically decreasing function of $a$ and this gives that 
\beas \frac{\pa}{\pa a}\varphi_3(a,\rho)\geq \frac{\pa}{\pa a}\varphi_3(1,\rho)=\frac{1}{1-\rho} - \frac{2\rho}{(1-\rho)^2}=\frac{1-3\rho}{(1-\rho)^2}\geq 0\quad\text{for} \quad\rho\leq 1/3.\eeas
Therefore, $\varphi_3(a,\rho)$ is a monotonically increasing function of $a\in[0,1]$ and it follows that 
\beas \varphi_3(a,\rho)\leq \varphi_3(1,\rho)=\frac{1}{1-\rho}\quad\text{for}\quad\rho\leq \frac{1}{3}.\eeas
\indent To prove the sharpness of the result, we consider the function $f_3(z)$ in $\Omega_{\gamma}$ such that $f_3=\psi\circ\Phi_3$, where $\Phi_3 : \Omega_{\gamma}\to\mathbb{D}$ defined by $\Phi_3(z)=\gamma+(1-\gamma)z$ and $\psi : \mathbb{D}\to\mathbb{D}$ defined by $\psi(z)=(a-z)/(1-az)$ for $a\in(0,1)$ and $\gamma\in[0,1)$. Therefore,
\beas f_3(z)
=A_0-\sum_{n=1}^\infty A_n \left(z+\frac{\gamma}{1-\gamma}\right)^n\quad\text{for}\quad a\in(0,1),\quad \gamma\in[0,1)\quad\text{and}\quad z\in\Omega_{\gamma},\eeas	
where $A_0=a$ and $A_n=a^{n-1}(1-a^2)(1-\gamma)^n$ for $n\geq 1$. Thus, we have
\beas\sum_{n=0}^\infty \left(\sum_{k=0}^n \frac{|A_k|}{(1-\gamma)^n}\right) \rho^n&=&\sum_{n=0}^\infty a\rho^n+\sum_{n=0}^\infty \left(\sum_{k=1}^n \frac{|A_k|}{(1-\gamma)^n}\right)\rho^n\\[2mm]
&=&\frac{a}{1-\rho}+(1-a^2)\sum_{n=0}^\infty \left(\sum_{k=1}^n a^{k-1}\right)\rho^n\\[2mm]
&=&\frac{a}{1-\rho}+(1-a^2) \sum_{n=0}^\infty \frac{(1-a^n)}{(1-a)}\rho^n\\[2mm]
&=&\frac{1}{1-\rho}+\frac{G_3(a,\rho)}{1-\rho},\eeas
where $G_3(a,\rho)=2a-(1+a)(1-\rho)/(1-a\rho)$.
Differentiating partially $G_3(a,\rho)$ with respect to $\rho$, we have
\beas\frac{\pa}{\pa \rho}G_3(a,\rho)=\frac{1-a^2}{(1-a\rho)^2}>0.\eeas
Therefore, $G_3(a,\rho)$ is a strictly increasing function of $\rho\in(0,1)$. Therefore, for $\rho>1/3$, we have 
\beas G_3(a,\rho)>G_3(a,1/3)=\frac{2(1-a)^2}{a-3}\quad\text{which tends to $0$ as}\quad a\to1^{-}.\eeas
Hence, $1/(1-\rho)+G_3(a,\rho)/(1-\rho)>1/(1-\rho)$ for $\rho>1/3$. This shows that $1/3$ is the best possible. This completes the proof.
\end{proof}
\section{ Statements}
\noindent{\bf Acknowledgment:} The work of the second author is supported by University Grants Commission (IN) fellowship (No. F. 44 - 1/2018 (SA - III)). The authors would like to thank the anonymous reviewers and the editing team for their valuable suggestions towards the improvement of the paper.\\[2mm]
{\bf Conflict of Interest:} The authors declare that there are no conflicts of interest regarding the publication of this
paper.\\[2mm]
{\bf Availability of data and materials:} Data sharing not applicable to this article as no data sets were generated or analysed during
the current study.

\end{document}